\documentclass[11pt]{amsart}

\usepackage{fullpage, graphicx, enumitem, mathtools,stmaryrd}
\SetSymbolFont{stmry}{bold}{U}{stmry}{m}{n}
\usepackage{url}
\usepackage{verbatim}
\usepackage{bbm}
\usepackage{amssymb,amsmath}
\usepackage{amsfonts,savesym,bm,amsthm}
\usepackage{xcolor}
\usepackage[mathscr]{eucal}
\usepackage[all]{xy}

\definecolor{hot}{RGB}{65,105,225}
\usepackage[pagebackref=true,colorlinks=true, linkcolor=hot ,  citecolor=hot, urlcolor=hot]{hyperref}

\newtheorem{theorem}{Theorem}[section]
\newtheorem{lemma}[theorem]{Lemma}
\newtheorem{conjecture}[theorem]{Conjecture}
\newtheorem{theorem-definition}[theorem]{Theorem-Definition}
\newtheorem{corollary}[theorem]{Corollary}
\newtheorem{proposition}[theorem]{Proposition}
\newtheorem{prop}[theorem]{Proposition}

\newtheorem{definition-theorem}[theorem]{Definition-Theorem}

\newtheorem{theorem-defintion}[theorem]{Theorem-Definition}
\newtheorem{corollary-definition}[theorem]{Corollary-Definition}
\newtheorem{definition-proposition}[theorem]{Definition-Proposition}

\theoremstyle{definition}

\newtheorem{definition}[theorem]{Definition}

\newtheorem{remark}[theorem]{Remark}

\newtheorem{subs}[theorem]{ }

\numberwithin{equation}{section}

\def\be{\begin{equation}}
	\def\ee{\end{equation}}

\title{A combinatorial criterion for detecting roots of Bernstein-Sato polynomials}

\author{Quan Shi}
\address{Department of Mathematical Sciences, Tsinghua University, Beijing, 100084, P. R. China.}
\email{shiq24@mails.tsinghua.edu.cn / thusq20@gmail.com}

\author{Huaiqing Zuo}
\address{Department of Mathematical Sciences, Tsinghua University, Beijing, 100084, P. R. China.}
\email{hqzuo@mail.tsinghua.edu.cn}

\begin{document}

	\begin{abstract}
		We give a combinatorial criterion for detecting roots of Bernstein-Sato polynomials of ideals.
		As applications, we prove the strong monodromy conjecture for determinantal and Pfaffian ideals.
		
		
	\end{abstract}
	
	\maketitle
	\tableofcontents

	\section{Introduction}
	
	Let $X$ be a smooth complex affine algebraic variety and let $\mathfrak a$ be a nonzero coherent
	ideal sheaf of $\mathcal O_X$. Let $f_1,\ldots,f_r$ be nonzero generators of $\mathfrak a$.
	The Bernstein-Sato polynomial of $\mathfrak a$, denoted by $b_{\mathfrak a}(s)$,
	is the monic polynomial in one variable of least degree such that
	\begin{equation}
		b_{\mathfrak a}(s_1+\cdots+s_r) f_1^{s_1}\cdots f_r^{s_r} \in \sum_{\bm u\in \mathbb Z^r, u_1+\cdots +u_r = 1}\mathscr D_X[s_1,\ldots,s_r] \cdot \Bigl(\prod_{u_i < 0}\binom{s_i}{-u_i}\Bigr) f_1^{s_1+u_1}\cdots f_r^{s_r+u_r}, \label{BS-intro}
	\end{equation}
	where $\mathscr D_X[s_1,\ldots,s_r]$ acts on $f_1^{s_1}\cdots f_r^{s_r}$ in the usual way.
	The polynomial $b_{\mathfrak a}(s)$ depends only on $\mathfrak a$, not on the choice of
	generators; see \cite{BMS06}. It is a fundamental invariant in $\mathscr D$-module theory,
	and its roots encode information about the singularities of $\mathfrak a$.
	
	In this paper, we give a combinatorial criterion for detecting these roots.

	Let $B=\mathbb Z[y_1,\ldots,y_n]$ and let $<$ be a monomial order on $B$.
	For a nonzero polynomial $F\in B$, let
	$\mathrm{in}(F)=c\cdot y_1^{a_1}\cdots y_n^{a_n}$ be its largest nonzero term with
	respect to $<$, and set $\mathrm{inv}(F):=(a_1,\ldots,a_n)\in\mathbb Z^n$.
	Let $\mathrm{mult}(F)$ denote the multiplicity of $F$ at $(y_1,\ldots,y_n)$. Equivalently, it is the least total degree of a monomial
	occurring in $F$.
	\begin{theorem}\label{general_statement}
		Let $B=\mathbb Z[y_1,\ldots,y_n]$, let $\mathfrak a$ be an ideal of $B$, and let $<$
		be a monomial order on $B$. Let $d\in\mathbb Z_{>0}$. Suppose that $\mathfrak a$ has
		generators $f_1,\ldots,f_r$ such that
		\begin{enumerate}
			\item $\mathrm{mult}(f_i) \geq d$ for all $i=1,\ldots,r$.\label{cA}
			\item There exists $r' \leq r$ such that $\deg \mathrm{in}(f_i) = d$ for all $i=1,\ldots,r'$, and the lattice $L = \sum_{i=1}^{r'} \mathbb Z \cdot \mathrm{inv}(f_i) \subset \mathbb Z^n$ has rank $n$.\label{cB}
		\end{enumerate} 
		Then $-\frac{n}{d}$ is a root of $b_{\mathfrak a}(s)$.
	\end{theorem}

	As applications of Theorem~\ref{general_statement}, we prove the strong monodromy conjecture
	for determinantal and Pfaffian ideals. This conjecture relates poles of motivic zeta functions,
	monodromy support loci, and roots of Bernstein-Sato polynomials.
	The motivic zeta function $Z_{\mathfrak a}^{\mathrm{mot}}(s)$ is the normalized generating
	series of classes of contact loci of $\mathfrak a$ in the Grothendieck ring of varieties.
	It is rational and can be computed using any log resolution of $\mathfrak a$; see \cite{DL98}.
	The monodromy support $S(\mathfrak a)$, also called the set of Verdier monodromy eigenvalues,
	generalizes the set of monodromy eigenvalues of local Milnor fibrations; see \cite{PV10}.

	\begin{conjecture}[strong monodromy conjecture]
		Let $X$ be a smooth algebraic variety and let $\mathfrak a$ be a coherent ideal sheaf
		of $\mathcal O_X$. Let $P(Z_{\mathfrak a}^{\mathrm{mot}})$ be the set of poles of
		$Z_{\mathfrak a}^{\mathrm{mot}}(s)$, and let $Z(b_{\mathfrak a})$ be the set of roots
		of $b_{\mathfrak a}(s)$. Then:
		\begin{enumerate}
			\item (Weak)\, $\exp(P(Z_{\mathfrak a}^{\mathrm{mot}})) \subset S(\mathfrak a)$, where $\exp : \mathbb C \to \mathbb C^*, s \mapsto e^{2\pi i s}$ is the exponential map.
			\item (Strong)\, $P(Z_{\mathfrak a}^{\mathrm{mot}}) \subseteq Z(b_{\mathfrak a})$.
		\end{enumerate}
	\end{conjecture}
	By the Kashiwara-Malgrange theorem, $\exp(Z(b_{\mathfrak a}))=S(\mathfrak a)$,
	so the strong conjecture implies the weak conjecture.

	Our applications to determinantal and Pfaffian ideals are as follows.
	\begin{theorem}\label{ThmA}
		Let $1 \leq r\leq n \leq m$, $R = \mathbb Z[x_{ij}]_{1\leq i\leq n,1\leq j\leq m}$, and $\mathfrak a_r$ be the ideal of $R$ generated by all $r\times r$-minors of $(x_{ij})_{i,j}$. Then the strong monodromy conjecture holds for $\mathfrak a_r$.
	\end{theorem}
	In this case, the poles of the motivic zeta function are known; see \cite{Roi,CZ25}.
	They are
	\begin{displaymath}
		\{-\frac{(m-j)(n-j)}{r-j} \mid 0\leq j \leq r-1\}.
	\end{displaymath}
	By induction, Theorem~\ref{ThmA} reduces to proving
	$b_{\mathfrak a_r}(-\frac{mn}{r})=0$; see \ref{reduction}.
	A similar reduction applies to Pfaffian ideals, for which we prove the following result.
	\begin{theorem}\label{ThmB}
		Let $1\leq 2r \leq N$, $R' := \mathbb Z[z_{ij}]_{1\leq i,j\leq N} / (z_{ij}+z_{ji},z_{ii}| 1\leq i,j \leq N) = \mathbb Z[z_{ij}]_{1\leq i<j\leq N}$, and $P_r^N$ be the ideal of $R'$ generated by all $2r \times 2r$-Pfaffians of $(z_{ij})_{i,j}$. Then the strong monodromy conjecture holds for $P_r^N$.
	\end{theorem}
	Recall the strong monodromy conjecture for determinantal ideals (resp. Pfaffian ideals) was already known when $r = n$ (resp. $r = \lfloor N/2 \rfloor$), cf. \cite{Bernstein_for_maximal_minors,CSS13}, while general cases remained open. The weak monodromy conjecture for determinantal and Pfaffian ideals was proved in
	\cite{CZ25} and \cite{CSXZ25}, respectively, using a detailed analysis of monodromy zeta
	functions. Our proof of the strong monodromy conjecture uses a different approach,
	based on a characteristic $p$ criterion of Musta\c{t}\u{a}, Takagi, and Watanabe \cite{MTW}.

	\begin{remark}[The use of AI]
		This work was assisted by AI. More precisely, ChatGPT Pro suggested an approach to
		proving Theorem~\ref{ThmA} when $(n,m,r)=(n,n,n-2)$. We corrected the argument and
		generalized it to obtain Theorem~\ref{general_statement}. All AI-generated content
		has been checked by the authors.
	\end{remark}
	\subs{\bf Acknowledgement.} Q. Shi and H. Zuo were supported by BJNSF Grant 1252009. Q. Shi was supported by NSFC Grant 125B2004. H. Zuo was supported by NSFC Grant 12271280.
	
		 %Throughout the paper, we will stick to the notation $r,m,n,R,\mathfrak a$.
	
	\section{Preliminaries}
	
	\subs{\bf Bernstein-Sato polynomial of an ideal.}\label{Bernstein-Sato}
	Let $X=\mathrm{Spec}\,B$ be a smooth complex affine variety, and let $\mathfrak a$ be an
	ideal of $B$ with nonzero generators $f_1,\ldots,f_r$. Then $b_{\mathfrak a}$ is the monic
	polynomial in one variable of least degree such that
	\begin{equation}
		b_{\mathfrak a}(s_1+\cdots+s_r) f_1^{s_1}\cdots f_r^{s_r} \in \sum_{\bm u\in \mathbb Z^r, u_1+\cdots+u_r = 1}\mathscr D_X[s_1,\ldots,s_r] \cdot \Bigl(\prod_{u_i < 0}\binom{s_i}{-u_i}\Bigr) f_1^{s_1+u_1}\cdots f_r^{s_r+u_r}, \label{BS}
	\end{equation}
	where $\mathscr D_X[s_1,\ldots,s_r]$ acts on $f_1^{s_1}\cdots f_r^{s_r}$ in the usual way.
	Replacing $X$ with an open affine subset $U\subset X$ in \eqref{BS} defines
	$b_{\mathfrak a|_U}(s)$, which divides $b_{\mathfrak a}(s)$. For $p\in X$, we define
	$b_{\mathfrak a,p}(s):=\mathrm{gcd}_{U\ni p}b_{\mathfrak a|_U}(s)$, where $U$ ranges
	over all open affine neighborhoods of $p$. This is the local Bernstein-Sato polynomial
	of $\mathfrak a$ at $p$.
	
	\begin{remark}\label{b-fun_var}
		Let $Z:=V(\mathfrak a)$ be the closed subscheme defined by $\mathfrak a$.
		By \cite[2.5]{BMS06}, $b_Z(s):=b_{\mathfrak a}(s-\mathrm{codim}_XZ)$ depends only
		on $Z$, not on its embedding into an affine space. Moreover, for a smooth affine
		variety $T$, \cite[Theorem 5]{BMS06} gives $b_Z(s)=b_{Z\times T}(s)$.
	\end{remark}

	\subs{\bf Reduction of Theorems~\ref{ThmA} and~\ref{ThmB}.}\label{reduction}
	Set $R_{n,m}:=\mathbb C[x_{ij}]_{1\leq i\leq n,1\leq j\leq m}$.
	Let $I_r^{n,m}$ be the ideal generated by all $r$-minors of the generic $n\times m$
	matrix, and let $Z_r^{n,m}$ be the closed subvariety defined by $I_r^{n,m}$.
	The closed points of $Z_r^{n,m}$ correspond to matrices of rank at most $r-1$.
	Let $U=\{x_{11}\neq0\}\subset X_{n,m}:=\mathrm{Spec}\,R_{n,m}$.
	The argument of \cite[Lemma 2.3]{Bernstein_for_maximal_minors} gives the following lemma.
	\begin{lemma}\label{local_det}
		Suppose $r \geq 2$. There exists an isomorphism of algebraic varieties
		\begin{displaymath}
			\pi : U \cap Z_{r}^{n,m} \overset{\simeq}{\longrightarrow} \mathbb C^* \times \mathbb C^{m-1} \times \mathbb C^{n-1} \times Z_{r-1}^{n-1,m-1}.
		\end{displaymath}	
		In particular, since $\mathrm{codim}_{X_{n,m}} Z_{r}^{n,m} = \mathrm{codim}_{X_{n-1,m-1}}Z_{r-1}^{n-1,m-1}$, we have $b_{I_{r-1}^{n-1,m-1}}(s)| b_{I_{r}^{n,m}}(s)$.
	\end{lemma}
	By Lemma \ref{local_det} and a simple induction, we have:
	\begin{prop}\label{mn_over_r}
		Suppose $-\frac{mn}{r}$ is a root of $b_{I_{r}^{n,m}}$ for all $1\leq r \leq n\leq m$, then $$-\frac{(m-j)(n-j)}{r-j} ,\ 0\leq j \leq r-1 $$ are all roots of $b_{I_{r}^{n,m}}(s)$.
	\end{prop}
	
	Similar results hold for Pfaffian ideals. Set $R'_N:=\mathbb C[z_{ij}]_{1\leq i<j\leq N}$.
	Let $W_r^N$ be the closed subvariety of $C_N:=\mathrm{Spec}\,R'_N$ defined by the
	complexification of $P_r^N$, and let $U'=\{z_{12}\neq0\}\subset C_N$.
	\begin{lemma}
		Suppose $r \geq 2$. There exists an isomorphism of algebraic varieties
		\begin{displaymath}
			\pi' : U' \cap W_{r}^N \overset{\simeq}{\longrightarrow} \mathbb C^* \times \mathbb C^{2(N-2)} \times W_{r-1}^{N-2}.
		\end{displaymath}
		In particular, since $\mathrm{codim}_{C_N} W_r^N = \mathrm{codim}_{C_{N-2}} W_{r-1}^{N-2}$, we have $b_{P_{r-1}^{N-2}}(s) \mid b_{P_r^N}(s)$.
	\end{lemma}
	\begin{proposition}\label{redu_Pf}
		Suppose $-\frac{N(N-1)}{2r}$ is a root of $b_{P_r^N}(s)$ for all $1\leq 2r \leq N$, then
		\begin{displaymath}
			-\frac{(N-2j)(N-2j-1)}{2(r-j)},\ 0\leq j \leq r-1
		\end{displaymath}
		are all roots of $b_{P_{r}^N}(s)$.
	\end{proposition}

	\subs{\bf A characteristic $p$ criterion.} 
	\begin{definition}
		For ideals $\mathfrak c_p,J_p\subseteq \mathbb F_p[y_1,\ldots,y_q]$ with $\mathfrak c_p \subset \sqrt{J_p}\neq (1)$ and $e \in \mathbb Z_{> 0}$, we define
		\begin{displaymath}
			v_{\mathfrak c_p}^{J_p}(p^e) = \max\{r\in\mathbb Z_{\geq0} \mid \mathfrak c_p^r \not\subseteq J_p^{[p^e]}\}.
		\end{displaymath} 
	\end{definition}
	Applying the proof of \cite[Proposition 3.11]{MTW}, we deduce the following.
	\begin{proposition}\label{MTWprop}
		Let $\mathfrak c,J \subseteq \mathbb Z[y_1,\ldots,y_q]$ be non-zero ideals such that $\mathfrak c \subset \sqrt{J}$ and $J \cdot \mathbb Q[y_1,\ldots,y_q] \neq \mathbb Q[y_1,\ldots,y_q] $. Then for $p \gg 0$ and $e \in \mathbb Z_{>0}$, we have
		\begin{displaymath}
			b_{\mathfrak c}(v_{\mathfrak c}^J(p^e)) \equiv 0\ (\mathrm{mod}\, p),
		\end{displaymath}
		where $\mathfrak c,J$ are identified with their reductions in $\mathbb F_p[y_1,\ldots,y_q]$.
	\end{proposition}
	Proposition~\ref{MTWprop} immediately yields the following criterion.
	\begin{corollary}\label{key_criterion}
		With notation as in Proposition~\ref{MTWprop}, let $\alpha\in\mathbb Q$.
		Suppose that there are infinitely many primes $p$ such that
		$\alpha\equiv v_{\mathfrak c}^{J}(p)\ (\mathrm{mod}\,p)$.
		Then $\alpha$ is a root of $b_{\mathfrak c}(s)$.
	\end{corollary}

	\section{Proof of Theorem \ref{general_statement}}
	
	With notation as in Theorem~\ref{general_statement}, let
	$K=(y_i^{A_i}\mid i=1,\ldots,n)$, where $A_i>0$.
	We construct the vector $\bm A:=(A_i)_i$ as follows.
	
	Recall $L = \sum_{i=1}^{r'} \mathbb Z \cdot \mathrm{inv}(f_i)$ and set
	\begin{displaymath}
		S = \sum_{i=1}^{r'} \mathbb N \cdot \mathrm{inv}(f_i),\ C=\sum_{i=1}^{r'}\mathbb R_{\geq0}\cdot \mathrm{inv}(f_i), \text{ and } \overline S:=C\cap L.
	\end{displaymath}

	Then there exists $c_0\in S$ such that $c_0+\overline S\subset S$; see e.g. \cite[Theorem 1.4]{KK12}.
	
	Since $L$ has rank $n$, the cone $C$ is full-dimensional.
	The vector $\bm B:=\sum_{i=1}^{r'}\mathrm{inv}(f_i)$ lies in $\mathrm{int}(C)$.
	Hence, for a sufficiently large positive integer $M$, the vector
	$\bm A:=M\bm B+\bm1$ also lies in $\mathrm{int}(C)$, where
	$\bm1=(1,\ldots,1)\in\mathbb Z^n$. For every sufficiently large prime $p$, we then have
	$p\bm A-\bm1-c_0\in\mathrm{int}(C)$.
	Since $\mathrm{rank}\,L=n$, there exists $k\in\mathbb Z_{>0}$ such that
	$(k\mathbb Z)^n\subseteq L$, and hence $(kd\mathbb Z)^n\subseteq L$.
	For $p\equiv1\ (\mathrm{mod}\,kd)$, we obtain
	$p\bm A-\bm1=pM\bm B+(p-1)\bm1\in L$.
	Set $A_{\mathrm{tot}}:=\sum_{i=1}^{n}A_i$ and $B_{\mathrm{tot}} := \sum_{i=1}^n B_i = \sum_{i=1}^{r'} \deg \mathrm{in}(f_i) = r'd$. Then
	$pA_{\mathrm{tot}}-n=\sum_{i=1}^{n}(pA_i-1) = p(MB_{\mathrm{tot}}+n)-n = pMr'd +n(p-1)$ is divisible by $d$.

	By Dirichlet's theorem, there are infinitely many such primes.
	Since $r_p \equiv -n/d\, (\mathrm{mod}\, p)$ for $p> d$, it therefore suffices to show that $v_{\mathfrak a}^K(p)=r_p:=(pA_{\mathrm{tot}}-n)/d$
	for these primes. On the one hand, every generator has multiplicity at least $d$,
	so every monomial occurring in an element of $\mathfrak a^{r_p+1}$ has degree at least
	$d(r_p+1)=pA_{\mathrm{tot}}-n+d>pA_{\mathrm{tot}}-n$.
	Thus $\mathfrak a^{r_p+1}\subset K^{[p]}=(y_i^{pA_i}\mid i=1,\ldots,n)$.
	On the other hand, for every sufficiently large prime $p\equiv1\ (\mathrm{mod}\,kd)$,
	we have $p\bm A-\bm1\in c_0+\overline S\subset S$.
	Hence there exist $\lambda_1,\ldots,\lambda_{r'}\in\mathbb N$ such that
	$\mathrm{in}(f_1^{\lambda_1}\cdots f_{r'}^{\lambda_{r'}})
	=c\cdot y_1^{pA_1-1}\cdots y_n^{pA_n-1}\notin K^{[p]}$
	for some $c\neq0$.
	Since $\deg\mathrm{in}(f_i)=d$ for all $i=1,\ldots,r'$ and $K^{[p]}$ is a monomial ideal,
	we have $f_1^{\lambda_1}\cdots f_{r'}^{\lambda_{r'}}\in
	\mathfrak a^{r_p}\setminus K^{[p]}$, proving the assertion. $\hfill\Box$
	
	\section{Proofs of Theorems \ref{ThmA} and \ref{ThmB}}

	\subs{\bf Determinantal ideals.}
	Let $1\leq r\leq n\leq m$, let $R=\mathbb Z[x_{ij}]_{1\leq i\leq n,1\leq j\leq m}$,
	and let $\mathfrak a_r$ be the ideal generated by all $r\times r$ minors of $(x_{ij})_{i,j}$.
	The case of maximal minors, $r=n$, was settled in \cite{Bernstein_for_maximal_minors},
	where the entire polynomial $b_{\mathfrak a_r}(s)$ was determined.
	For $r=1$, \cite[Theorem 1.1]{Mus22} gives $b_{\mathfrak a_1}(s)=s+mn$.
	We therefore assume $2\leq r<n\leq m$ from now on.

	We define a monomial order on $R$ as follows. First, order the variables $\{x_{ij}\}_{i,j}$
	and let $\prec$ be the induced lexicographic monomial order.
	For a monomial $A=\prod_{i,j}x_{ij}^{b_{ij}}$, set $\deg_0 A:=\sum_{i=1}^n b_{ii}$.
	For two monomials $A_1,A_2$, declare $A_1<A_2$ if:
	\begin{enumerate}
		%\item $\deg A_1 < \deg A_2$, or
		\item $\deg_0 A_1 < \deg_0 A_2$, or
		\item $\deg_0 A_1 = \deg_0 A_2$, and $A_1 \prec A_2$.
	\end{enumerate}
	This defines a monomial order $<$ on $R$.

	For simplicity, we set $[n] = \{1,\ldots,n\}$ and $[m] = \{1,\ldots,m\}$. For $I_1\subseteq [n],I_2 \subseteq [m]$ with $\vert I_1\vert = \vert I_2\vert = c$, denote by $\det X_{I_1,I_2} = \det (x_{ij})_{i\in I_1,j\in I_2}$ the $c\times c$-minor determined by the rows indexed by $I_1$ and the columns indexed by $I_2$. Now for $I \subset [n]$ with $\vert I \vert = r$, we set $D_I := \det X_{I,I}$. For $J \subseteq [n]$ with $\vert J\vert = r-1$, $i \in [n]\setminus J$, and $j\in [m]\setminus (J\cup \{i\})$, we set $G_{ij}^J := \det X_{J\cup \{i\},J\cup \{j\}}$. Then we have
	\begin{displaymath}
		\mathrm{in}(D_I) = \prod_{l\in I} x_{ll}\text{ and } \mathrm{in}(G_{ij}^J) = \pm x_{ij} \prod_{l\in J} x_{ll}.
	\end{displaymath}
	Let $\{e_{ij}\}$ be the standard basis of $\mathbb Z^{n \times m}$. For $I \subset [n]$ with $\vert I \vert = r$, we set $d_I := \sum_{l\in I} e_{ll}$. For $J \subseteq [n]$ with $\vert J\vert = r-1$, $i \in [n]\setminus J$, and $j\in [m]\setminus (J\cup \{i\})$, we set $g_{ij}^J := e_{ij}+\sum_{l\in J} e_{ll}$. Then $\mathrm{inv}(D_I) = d_I$ and $\mathrm{inv}(G_{ij}^J) = g_{ij}^J$ for all $I,J,i,j$.

	Now take $\big(B,\mathfrak a,d,(f_1,\ldots,f_{r'})\big)$ in Theorem \ref{general_statement} to be $\big(R,\mathfrak a_r,r,(D_I,G_{ij}^J \mid I,J,i,j)\big)$. Then $L$ is the lattice generated by all $d_I,g_{ij}^J$.
	\begin{lemma}\label{ful_det}
		We have $L = \{\sum_{i,j} \lambda_{ij} e_{ij} \mid \lambda_{ij} \in \mathbb Z, \sum_{i,j} \lambda_{ij} \text{ is divisible by }r\}$. In particular, the rank of $L$ is $mn$.
	\end{lemma}
	\begin{proof}
		Let $
		L_0=\left\{
		\sum_{i,j}\lambda_{ij}e_{ij}
		\;\middle|\;
		\lambda_{ij}\in\mathbb Z,\quad
		r\text{ divides }\sum_{i,j}\lambda_{ij}
		\right\}.$ Since every $d_I$ and $g_{ij}^J$ belongs to $L_0$, we have
		$L\subseteq L_0$. Set $A_0:=\mathbb Z^{n\times m}/L$, then $\vert A_0\vert \geq |\mathbb Z^{n\times m}/L_0| = r$.
		To prove the reverse inclusion, it suffices to show that $A_0$
		is finite and $	\vert A_0\vert \leq r$.

		For distinct $i,j\in[n]$, choose
		$I_0\subseteq[n]\setminus\{i,j\}$ with $|I_0|=r-1$.
		Such a choice is possible since $r<n$. Then $
		e_{ii}-e_{jj}
		=d_{I_0\cup\{i\}}-d_{I_0\cup\{j\}}\in L.$ Thus all diagonal basis vectors have the same class $z\in A_0$.
		Since $d_{\{1,\ldots,r\}}\in L$, we have $rz=0\in A_0$. For $i\in[n]$ and $j\in[m]$ with $i\neq j$, choose
		$J\subseteq[n]\setminus\{i,j\}$ with $|J|=r-1$.
		Writing $\overline{e}_{ij}$ for the class of $e_{ij}$ in $A_0$,
		the relation $g_{ij}^J\in L$ gives
		\[
		0=\overline{e}_{ij}+(r-1)z
		=\overline{e}_{ij}-z.
		\]
		Hence every standard basis vector has class $z$, so $A_0$
		is cyclic and generated by $z$. Since $rz=0$, we obtain
		$|A_0|\leq r$, completing the proof.
	\end{proof}
	\subs{\bf Proof of Theorem \ref{ThmA}.}
	It suffices to verify the two conditions of Theorem~\ref{general_statement}.
	Condition~\eqref{cA} holds because all $r\times r$ minors are homogeneous of degree $r$,
	and condition~\eqref{cB} follows from Lemma~\ref{ful_det}. $\hfill\Box$

	\subs{\bf Pfaffian ideals.}
	Let $1\leq2r\leq N$, and set
	$R':=\mathbb Z[z_{ij}]_{1\leq i,j\leq N}/(z_{ij}+z_{ji},z_{ii}\mid 1\leq  i,j \leq N)
	=\mathbb Z[z_{ij}]_{1\leq i<j\leq N}$.
	Let $P_r^N$ be the ideal generated by all $2r\times2r$ Pfaffians of $(z_{ij})_{i,j}$.
	We prove the strong monodromy conjecture for $P_r^N$.

	For $r=1$, \cite{Mus22} gives $b_{P_1^N}(s)=s+N(N-1)/2$, which proves the conjecture
	in this case. When $N$ is odd and $r=\lfloor N/2\rfloor$, the result follows from
	\cite{Bernstein_for_maximal_minors}. When $N$ is even and $r=N/2$, the conjecture follows
	from $b_{P_r^{2r}}(s)=(s+1)(s+3)\cdots(s+2r-1)$; see \cite[\S 9]{CSS13}.
	Note that the polynomial $b(s)$ in \cite{CSS13} equals $b_{P_r^{2r}}(s-1)$.

	We may therefore assume $2\leq r<\lfloor N/2\rfloor$.
	By \cite{CSXZ25}, the poles of $Z_{P_r^N}^{\mathrm{mot}}(s)$ are
	\begin{displaymath}
		-\frac{N(N-1)}{2r}, -\frac{(N-2)(N-3)}{2(r-1)},\ldots,-\frac{(N-2r+2)(N-2r+1)}{2}.
	\end{displaymath}
	By Proposition~\ref{redu_Pf}, it remains to show that
	$-\frac{N(N-1)}{2r}$ is a root of $b_{P_r^N}(s)$.
	
	Set $h:=\lfloor N/2\rfloor$ and $T:=\{1,3,\ldots,2h-1\}$.
	We distinguish the pairwise disjoint pairs $\{a,a+1\}$ for $a\in T$.
	For a monomial $B=\prod_{i<j}z_{ij}^{b_{ij}}$, define
	\[
	\deg'_0 B:=\sum_{a\in T}b_{a,a+1}.
	\]
	
%	For a monomial $B = \prod_{i<j} z_{ij}^{b_{ij}}$,  we define
%	\begin{displaymath}
%		\deg_{0}' B := \sum_{i=1}^{N-1} b_{i,i+1}.
%	\end{displaymath}

	Fix an arbitrary monomial order $\prec$ on $R'$.
	For two monomials $A_1,A_2$ in $R'$, declare $A_1<A_2$ if:
	\begin{enumerate}
		\item $\deg_0' A_1 < \deg_0' A_2$, or
		\item $\deg_0' A_1 = \deg_0' A_2$, and $A_1 \prec A_2$.
	\end{enumerate}
	This defines a monomial order $<$ on $R'$.

	For $U\subseteq T$, define $U+1:=\{a+1\mid a\in U\}$.
	For $I\subseteq T$ with $|I|=r$, define
	\[
	H_I:=\mathrm{Pf}(z_{ij})_{i,j\in I\cup(I+1)}.
	\]
	For $J\subseteq T$ with $|J|=r-1$ and indices $i<j$
	outside $J\cup(J+1)$, define
	\[
	Q_{ij}^J:=\mathrm{Pf}(z_{kl})_{k,l\in J\cup(J+1)\cup\{i,j\}}.
	\]
	Then
	\[
	\mathrm{in}(H_I)=\prod_{a\in I}z_{a,a+1}
	\quad\text{and}\quad
	\mathrm{in}(Q_{ij}^J)=z_{ij}\prod_{a\in J}z_{a,a+1}.
	\]

	Let $\{\tilde e_{ij}\}_{1\leq i<j\leq N}$ be the standard basis of
	$\mathbb Z^{N(N-1)/2}$. For $I,J,i,j$ as above, set
	$h_I:=\sum_{l\in I}\tilde e_{l,l+1}$ and
	$q_{ij}^J:=\tilde e_{ij}+\sum_{l\in J}\tilde e_{l,l+1}$.

	Now take $\big(B,\mathfrak a,d,(f_1,\ldots,f_{r'})\big)$ in Theorem \ref{general_statement} to be $\big(R',P_r^N,r,(H_I,Q_{ij}^J \mid I,J,i,j)\big)$. Then $L$ is the lattice generated by all $h_I,q_{ij}^J$.
	We have the following analogue of Lemma~\ref{ful_det}.
	\begin{lemma}\label{ful_Pf}
		$L = \{\sum_{i<j} \lambda_{ij} \tilde e_{ij} \mid \lambda_{ij} \in \mathbb Z,\ r\text{ divides }\sum_{i<j} \lambda_{ij}\}$. In particular, the rank of $L$ is $N(N-1)/2$.
	\end{lemma}

	\subs{\bf Proof of Theorem \ref{ThmB}.}
	By Proposition~\ref{mn_over_r}, it suffices to verify the two conditions of Theorem~\ref{general_statement}.
	Condition~\eqref{cA} holds because all $2r\times2r$ Pfaffians are homogeneous of degree $r$,
	and condition~\eqref{cB} follows from Lemma~\ref{ful_Pf}. $\hfill\Box$
	
	%\section{Other applications}

	%	


\begin{thebibliography}{BBNV25 } 
		%		
		%		%\bibitem[AC75]{AC} N. A’Campo. \newblock{La fonction z\^eta d’une monodromie.} {\em Comment. Math. Helv.} 50: 233-248, 1975.
		%		
		%		\bibitem[ACLM05]{Artal} E. Artal Bartolo, P. Cassou-Noguès, I. Luengo, A. Melle Hernández. Quasi-ordinary power series and their zeta functions. Mem. Amer. Math. Soc. 178, vi+85 pp. (2005). 
		%		
		%		\bibitem[BBNV25]{Birational zeta function}
		%		T. Biesbrouck, N. Budur, J. Nicaise,  W. Veys.
		%		\newblock Birational zeta functions.
		%		\newblock{arxiv:2509.03352}.
		%		
		%		
		%		\bibitem[BVVW24]{BVVW24}
		%		N. Budur, R. van der Veer, and A. Van Werde.
		%		Estimates for zero loci of Bernstein-Sato ideals. 
		%		\newblock{\em Publ. Res. Inst. Math. Sci.}, 60, 1-22 (2024).
		
		%		%\bibitem[Bla22]{Bla22} G. Blanco.	\newblock{An algorithm for Hodge ideals}. {\tt hodge.lib}.		\newblock{\em Math. Comp.}, 338: 2955--2967.
		%		
		%		\bibitem[Bla24]{Bla24}
		%		G. Blanco.
		%		Topological roots of the Bernstein-Sato polynomial of plane curves.
		%		\newblock{arXiv:2406.09034.}
		%		\bibitem[BK86]{BK} E. Brieskorn, H. Kn\"orrer. 
		%		{Plane algebraic curves.} Birkh\"auser/Springer,  x+721 pp. (1986).
		%		
		%		\bibitem[B03]{Bu} N. Budur. On Hodge spectrum and multiplier ideals.  Math. Ann.  327,  257-270 (2003). 
		%		
		%		
		%		\bibitem[B12]{Bsurvey} N. Budur. Singularity invariants related to Milnor fibers: survey. Contemp. Math. 566, 161-187 (2012).
		%		
		%		\bibitem[BGG12]{BGV}	N. Budur, P. González-Pérez, M. González Villa. 
		%		Log canonical thresholds of quasi-ordinary hypersurface singularities.
		%		Proc. Amer. Math. Soc. 140, 4075-4083 (2012).
		%		
		%		\bibitem[B+24]{BBPP24}
		%		N. Budur, J. de la Bodega, E. de Lorenzo Poza, J. Fern\'andez de Bobadilla, T. Pe\l ka.
		%		\newblock{On the embedded Nash problem}.
		%		\newblock{Forum Math. Pi} 12, Paper No. e15, 28 pp. (2024).
		
		%		\bibitem[BPSZ26]{BPSZ} N. Budur, E. de Lorenzo Poza, Q. Shi, and H. Zuo. \newblock{A remarkable subset of poles of the motivic zeta function}. \newblock{arXiv:2602.13508.}
		
		  \bibitem[BMS06]{BMS06}
		N.~Budur, M.~Musta\c{t}\u{a}, and M.~Saito.
		\newblock Bernstein-{S}ato polynomials of arbitrary varieties.
		\newblock {\em Compos. Math.}, 142(3): 779--797, 2006.
		
		%		
		%		
		%		\bibitem[BFLN22]{Cohomology_of_Contact_Loci}
		%		N. Budur, J.~Fern\'andez de~Bobadilla, Q.T. L\^e,  H.D.
		%		Nguyen.
		%		\newblock Cohomology of contact loci.
		%		\newblock {J. Differential Geom.} 120, 389-409 (2022).
		%		
		%		\bibitem[BPSZ26]{BPSZ} N. Budur, E. de Lorenzo Poza, Q. Shi,  H. Zuo. \newblock{A remarkable subset of poles of the motivic zeta function}. \newblock{ arXiv:2411.01305.}
		%		
		%		\bibitem[BSZ24]{BSZ} N. Budur, Q. Shi,  H. Zuo. \newblock{Motivic principal value integrals for hyperplane arrangements}. \newblock{ arXiv:2411.01305.}
		%		
		%
		%		
		%		\bibitem[BN20]{BN20}
		%		E. Bultot, J. Nicaise.
		%		\newblock{ Computing motivic zeta functions on log smooth models.}
		%		\newblock{ Math. Zeit.} 295, 427-462 (2020).
		%		
		%		\bibitem[BL23]{BL}
		%		J. de la Bodega, E. de Lorenzo Poza.
		%		\newblock{The arc-Floer conjecture for plane curves}.
		%		\newblock{arXiv:2308.00051}. To appear in J. Differential Geom.
		%		
		
				\bibitem[CSS13]{CSS13}
				S. Caracciolo, A.~D. Sokal, and A. Sportiello.
				\newblock Algebraic/combinatorial proofs of Cayley-type identities for derivatives of determinants and Pfaffians.
				\newblock{\em Adv. in Appl. Math.}, 50(4), 474--594, 2013.
		
				\bibitem[CZ25]{CZ25}
				Y. Chen and H. Zuo.
				\newblock{On the monodromy conjecture for determinantal varieties}.
				\newblock{\em \url{https://arxiv.org/abs/2510.11425}}, 2025.
				
				\bibitem[CSXZ25]{CSXZ25}
				Y. Chen, Q. Shi, Y. Xu, and H. Zuo.
				\newblock{On the monodromy conjecture, holomorphy conjecture, and embedded Nash problem for Pfaffian ideals.}
				\newblock{\url{https://arxiv.org/abs/2511.01262}}, 2025.
				
				
				\bibitem[DL98]{DL98}
				J. Denef and F. Loeser.
				\newblock Motivic {I}gusa zeta functions.
				\newblock { J. Algebraic Geom.} 7, 505-537, 1998.

		%		
		%		%\bibitem[DL99]{DL99} J. ~Denef,  F. ~Loeser. \newblock {Germs of arcs on singular algebraic varieties and motivic integration}.	\newblock {\em Invent. Math.}, 135(1):201--232, 1999.
		%		
				\bibitem[Doc13]{Roi}
				R. Docampo. 
				\newblock{Arcs on determinantal varieties.}
				\newblock{\em Trans. Amer. Math. Soc.}, 365: 2241-2269, 2013.
				
		%		\bibitem[ELM04]{ELM04}
		%		L. Ein, R. Lazarsfeld, M. Musta\c t\u a.
		%		\newblock{ Contact loci in arc spaces.}
		%		\newblock{ Compos. Math.} 140, 1229-1244 (2004).
		%		
		%		
		%		
		%		\bibitem[FP24]{BP} J. Fern\'andez de Bobadilla, T. Pe\l ka. Symplectic monodromy at radius zero and equimultiplicity of $\mu$-constant families. { Ann. Math.} 200, 153-299 (2024).
		%		
		%		
		%		\bibitem[Fi12]{Fic} G. Fichou. The motivic real Milnor fibres. Manuscripta Math. 139, 167–178 (2012).
		%		
		%		
		%		
		%		\bibitem[GG14]{GV}  P. González Pérez, M. González Villa.
		%		Motivic Milnor fiber of a quasi-ordinary hypersurface. J. Reine Angew. Math. 687, 159-205 (2014).		
		%		
		%		\bibitem[GR25]{GP} P. González Pérez, M. Robredo Buces. Quasi-ordinary hypersurfaces, multiplier ideals and local tropicalizations. arXiv:2510.21009.
		%		
		%		
		%		\bibitem[I18]{Ish14}
		%		S. Ishii.
		%		\newblock{Introduction to singularities}.
		%		\newblock Springer, 2018, x+236 pp.
		%		
		%		\bibitem[KT19]{KT}
		%		M. Kontsevich, Y. Tschinkel.
		%		\newblock{Specialization of birational types.}
		%		\newblock{ Invent. Math.} 217, 415-432 (2019).
		%		
		%		
		%		\bibitem[Li89]{Li} B. Lichtin. Poles of $|f(z,w)|^{2s}$ and roots of the $b$-function. Ark. Mat. 27, 283-304 (1989).
		%		
		%		\bibitem[Lo88]{L} F. Loeser. Fonctions d'Igusa $p$-adiques et polyn\^{o}mes de
		%		Bernstein. Amer. J. Math. 110, 1-21 (1988).
		%		
		%		
		
		\bibitem[KK12]{KK12}
		K. Kaveh and A. G. Khovanskii.
		\newblock Newton--Okounkov bodies, semigroups of integral points, graded algebras and intersection theory.
		\newblock {\em Ann. of Math. (2)}, 176(2): 925--978, 2012.

		\bibitem[LRWW17]{Bernstein_for_maximal_minors}
		A. C. Lőrincz, C. Raicu, U. Walther, and J. Weyman.
		\newblock{Bernstein–Sato polynomials for maximal minors and sub-maximal Pfaffians}.
		\newblock{\em Adv. Math.}, 307: 224-252, 2017.
		%		
		%		
		%		\bibitem[Ma20]{Man} F. Mangolte. Real algebraic varieties.
		%		Springer, 2020, xviii+444 pp.	
		%		
		%		\bibitem[McCP03]{McP} C. McCrory, A. Parusi\'nski. Virtual Betti numbers of real algebraic varieties. C. R. Math. Acad. Sci. Paris 336, 763–768 (2003).
		%		
		%		\bibitem[McL19]{McL19}
		%		M. McLean.
		%		\newblock Floer cohomology, multiplicity and the log canonical threshold.
		%		\newblock { Geom. Topol.} 23, 957-1056 (2019).	
		%		
		%		
		%		
		%		
		%		
		%		\bibitem[Mu02]{Mus02}
		%		M. Musta\c t\u a.
		%		\newblock{ Singularities of pairs via jet schemes.}
		%		\newblock{ J. Amer. Math. Soc.} 15, 599-615 (2002).
		
		\bibitem[Mus22]{Mus22}
		M.~Musta\c{t}\u{a}.
		\newblock Bernstein-{S}ato polynomials for general ideals vs. principal ideals.
		\newblock {\em Proc. Amer. Math. Soc.}, 150(9): 3655--3662, 2022.
		
		
		\bibitem[MTW05]{MTW}
		M. Musta\c t\u a, S. Takagi, and K. Watanabe, F-thresholds and Bernstein-Sato polynomials, in {\it European Congress of Mathematics}, Eur. Math. Soc., Z\"urich, 341--364, 2005.
		
		
		%		
		%		\bibitem[NX16]{NX16}
		%		J. Nicaise, C. Xu.
		%		\newblock Poles of maximal order of motivic zeta functions.
		%		Duke Math. J. 165, 217-243 (2016).
		%		
		%		\bibitem[RV03]{RV} B. Rodrigues and W. Veys.
		%		Poles of zeta functions on normal surfaces.
		%		{ Proc. London Math. Soc.} 87, 164-196, 2003.
		
		%		
		%		\bibitem[S00]{Sexp} M. Saito. Exponents of an irreducible plane curve singularity. arXiv:math/0009133.
		%		
		%		
		%		\bibitem[S07]{S-real} M. Saito.  On real log canonical thresholds. arXiv:0707.2308.
		%		
		%		
		%		\bibitem[Va76]{Var76}
		%		A. Varchenko.
		%		\newblock Zeta-function of monodromy and Newton's diagram.
		%		\newblock{ Invent. Math.} 37, 253-262 (1976).
		%		
		%		\bibitem[Va82]{Va} A. Varchenko. The complex singularity index does not change along the stratum $\mu=$const. Funct. Anal. Appl. 16, 1-9  (1982).		
		%		
		
		%		\bibitem[Vey90]{Vey90}
		%		W. Veys.
		%		\newblock On the poles of Igusa’s local zeta function for curves.
		%		\newblock{J. Lond. Math. Soc. (2)}, 41(1), 27-32 (1990). 
		
		%		\bibitem[Ve95]{Ve} W. Veys. Determination of the poles of the topological zeta function for curves. Manuscripta
		%		Math 87, 435-448 (1995).		
		%		
		
		\bibitem[VPV10]{PV10}
		L. Van Proeyen and W. Veys. 
		\newblock{The monodromy conjecture for zeta functions associated to ideals in dimension two}.
		\newblock{\em Ann. Inst. Fourier (Grenoble)}, 60(4): 1347-1362, 2010.
		
	\end{thebibliography}
\end{document}